\documentclass[11pt,a4paper]{article}
\usepackage[dvipsnames]{xcolor}
\usepackage[all,2cell,cmtip]{xy}
\UseTwocells
\entrymodifiers={+!!<0pt,\fontdimen22\textfont2>}

\usepackage{fouriernc}  
\usepackage[english]{babel}  
\usepackage[T1]{fontenc}
\usepackage[utf8]{inputenc}
\usepackage{verbatim} 
\usepackage{amsmath}
\usepackage{breqn}
\usepackage{mathtools}
\usepackage{amsmath}
\usepackage{float}
\usepackage{amsfonts}
\usepackage{amssymb}
\usepackage{mathrsfs}    
\usepackage{amsthm} 

\usepackage{smartdiagram}
\usepackage{mathtools}
\usepackage[colorlinks=true,linkcolor=MidnightBlue,citecolor=OliveGreen]{hyperref}           

\usepackage{pgf,tikz}
\usetikzlibrary{arrows}
\usetikzlibrary{angles,quotes}
\usetikzlibrary{calc}
\usetikzlibrary{positioning}
\usetikzlibrary{calc,shapes.geometric}
\usetikzlibrary{decorations.pathreplacing}
\usetikzlibrary{cd}
\usetikzlibrary{positioning}
\hypersetup{bookmarksnumbered=true}
\usepackage{caption}
\usepackage{subcaption}

\newtheorem{theorem}{Theorem}

\theoremstyle{plain}
\newtheorem{thm0}{Theorem}
\newtheorem{cor0}{Corollary}

\newtheorem{thm}{Theorem}[subsection]

\newtheorem{lemma}[thm]{Lemma}

\newtheorem{axiom}[thm]{Axiom}
\newtheorem{question}[thm]{Question}

\theoremstyle{definition}
\newtheorem{def0}{Definition}
\newtheorem{axi0}{Axiom}

\theoremstyle{remark}
\newtheorem{remark}[thm]{Remark}
\newtheorem{eg}[thm]{Example}

\newcommand{\proofof}[1]{\end{#1}\begin{proof}}

\makeatletter
\renewcommand\section{\@startsection {section}{1}{\z@}%
  {-3.5ex \@plus -1ex \@minus -.2ex}{2.3ex \@plus.2ex}%
  {\normalfont\large\bfseries}}
\renewcommand\subsection{\@startsection{subsection}{2}{\z@}%
  {-3.25ex\@plus -1ex \@minus -.2ex}{1.5ex \@plus .2ex}%
  {\normalfont\bfseries}}
\DeclareMathAlphabet{\mathrmsl}{OT1}{cmr}{m}{sl}

\newcommand{\rssymb}[2]{\newcommand{#1}{\mathrmsl{#2}} }
\newcommand{\oper}[3][n]{\newcommand{#2}{\mathop{\mathrm{#3}}%
\ifx n#1\nolimits\else\limits\fi} }
\newcommand{\rsoper}[3][n]{\newcommand{#2}{\mathop{\mathrmsl{#3}}%
\ifx n#1\nolimits\else\limits\fi} }

\newcommand\doi[1]{\href{http://dx.doi.org/#1}{\texttt{doi:#1}}}

\usetikzlibrary{positioning, fit, calc}   
\usepackage{xcolor}  
\tikzset{block/.style={draw, thick, text width=2cm , minimum height=1.3cm, align=center},   
line/.style={-latex}     
}  

\oper\Ad{Ad}
\oper\ad{ad}

\oper\val{val}
\oper\coker{coker}
\oper\mult{mult}
\oper\Iso{Iso}
\oper\End{End}
\oper\Aut{Aut}
\oper\Sub{Sub}
\oper\Alt{Alt}
\oper\Ext{Ext}
\oper\Pic {Pic}
\oper\Sym{Sym}
\oper\Spec{Spec}
\oper\Spf{Spf}
\oper\Sp{Sp}
\oper\Spa{Spa}
\oper\Proj{Proj}
\rsoper\divg{div}
\rsoper{\sym}{sym}
\rsoper{\alt}{alt}
\rsoper\trace{tr}
\rssymb\id{id}

\newcommand{\thismonth}{\ifcase\month\or
  January\or February\or March\or April\or May\or June\or
  July\or August\or September\or October\or November\or December\fi
  \space\number\year}
  
\makeatletter
\newcommand{\address}[1]{\gdef\@address{#1}}
\newcommand{\email}[1]{\gdef\@email{\url{#1}}}
\newcommand{\@endstuff}{\par\vspace{\baselineskip}\noindent\small
\begin{tabular}{@{}l}\scshape\@address\\\textit{E-mail address:} \@email\end{tabular}}
\AtEndDocument{\@endstuff}
\makeatother
\address{MFO Oberwolfach}
\email{tewari@math.tu-berlin.de}
\usepackage{lipsum}

\usepackage{enumitem}

\makeatletter
\newcommand{\subjclass}[2][2020]{%
  \let\@oldtitle\@title%
  \gdef\@title{\@oldtitle\footnotetext{#1 \emph{Mathematics subject classification.} #2}}%
}
\newcommand{\keywords}[1]{%
  \let\@@oldtitle\@title%
  \gdef\@title{\@@oldtitle\footnotetext{\emph{Key words and phrases.} #1.}}%
}
\makeatother

\keywords{tropical collineation}

\usepackage{calc}

\usepackage{todonotes}

\title{Fundamental Theorem of Projective Geometry over Semirings}
\author{Ayush Kumar Tewari}

\subjclass{14N05, 16Y60, 14T10}
\keywords{linear semiring, semiring collineations, tropical perspectivity, tropical projectivity}
\begin{document}

\maketitle

\begin{abstract}
We state the fundamental theorem of projective geometry for semimodules over semirings, which is facilitated by recent work in the study of bases in semimodules defined over semirings. In the process we explore in detail the linear algebra setup over semirings. We also provide more explicit results to understand the implications of our main theorem on maps between tropical lines in the tropical plane. Along with this we also look at geometrical connections to the rich theory of tropical geometry. 
\end{abstract}

\tableofcontents

\section{Introduction}

Maps between projective spaces and their underlying algebraic structures have been a classical area of study in algebraic geometry. This has been immensely enriched with the seminal work of Tits on the theory of buildings and the connections with the study of classical groups. One of the prominent results in this case is the well known fundamental theorem of projective geometry which is a structure theorem and elaborates on the structure of the collinearity preserving maps between projective spaces associated with vector spaces over fields. This in turn also helps in assessing the automorphism group of such spaces and the group of isomorphisms of the underlying field. The theorem has appeared in various forms throughout the literature in the 19th and 20th century, with various attributions attached to the different versions. The first occurrence of some versions of the statement can be dated back to the work of Darboux, Segre, Veblen and Von Staudt \cite{von1847geometrie}, \cite{darboux}, \cite{segre1890nuovo}.  

Since then considerable work has been done to understand these maps and their outcomes in terms of the associated geometry. Various versions of the statement have been proved for linear systems over algebraic structures other than fields, for e.g. free modules over commutative rings \cite{ojanguren1969note}. Our goal in this article is to understand these results in the realm of semirings and semimodules over them. The notion of linear independence and basis elements is a bit nuanced in this setting. The one stark difference in this setting is that two basis of a semimodule over a semiring may not have the same number of elements. We use recent findings about the properties of basis over semimodules \cite{YJT14} and develop a background to define the notion of semiring semi-linear maps and semiring collineations.


\begin{center}
\begin{tikzpicture}  
\node[fill=olive] (a) {Modules (Vector Spaces)};  
\node[right=4cm of a, fill=yellow] (b) {Semimodules};   
\node[below=6cm of a, fill=green] (d) {Rings (Fields)};  
\node[right=5.1cm of d, fill=orange] (c) {Semirings};  
     ysep=6mm,fit=(d)(e),label={130:A}](f){}; 
\draw[line][dotted] (a)-- (b);  
\draw[line][dotted] (d)-- (c);  
\draw[line] (a)-- node[sloped, above, text width=5.0cm] {Fund. Theorem of Projective} node[sloped, below, text width=5.0cm] {Geometry}(d);  
\draw[line] (b)-- node[sloped, above, text width=5.0cm] {Fund. Theorem of (Tropical)} node[sloped, below, text width=5.0cm] {Projective Geometry} (c);  
  
\end{tikzpicture}  
\end{center}

Our goals in this work are twofold; firstly we describe an algebraic setup of semimodules over semirings along with a geometric aspect to understand how maps act on "lines" in "planes", with special emphasis in the case when the underlying semiring is the \emph{tropical semiring}. One of our first observations is that \emph{collinearity} is geometrically well represented by \emph{coaxiality} in the tropical setup. In Section \ref{sect:Section2}, we discuss the basic background of linear algebra over semirings, and describe the results regarding bases of semimodules over \emph{linear} semirings. In Section \ref{sect:Section3} we establish our main contribution in the form of Theorem \ref{thm:tropicalfund}, which we state as the fundamental theorem of projective geometry over semirings, which can be also understood as a fundamental theorem of tropical projective geometry. In Section \ref{sect:section4}, we see a more explicit version where we prove a variant of the fundamental theorem, restricted to the case of the tropical plane.

In Sections \ref{sec:section5} we discuss possible generalizations of our results in terms of a possible theory of buildings and complexes for tropical flag manifolds and flag varieties, which is inspired by some recent work in this field \cite{brandt2021},\cite{bao2020flag}. We also recognize that the theory of finitely generated semimodules ties up with \emph{tropical convexity}, as finitely generated semimodules are also considered as \emph{tropical cones} in the literature \cite{maclagan2009introduction},\cite{joswig2014essentials}. We introduce some of the connections of our work to the theory of tropical convexity and are pretty hopeful to take it further in subsequent work. We also highlight the fact that our results also complements the study on matrix semigroups over semirings \cite{matrix2018tropical}, \cite{gould2020matrix} and can have fruitful connections to problems concerning semigroups. We close our discussion with possible definitions of \emph{tropical cross ratios} which might be suitable in order to study tropical projective geometry further.

\paragraph{Acknowledgements}
I would like to thank Marianne Johnson and Hannah Markwig who provided their valuable comments and suggestions on this draft. I would also like to thank Michael Joswig for sharing his thoughts on the topic with me. I am also grateful to the Mathematisches Forschungsinstitut Oberwolfach for providing excellent facilities and for their warm hospitality during my stay. This research was supported through the programme  "Oberwolfach Leibniz Fellows" by the Mathematisches Forschungsinstitut Oberwolfach in 2021.

\section{Linear Algebra over Semirings}\label{sect:Section2}

A semiring is defined as follows \cite{golan2013semirings},

\begin{def0}
A semiring is a set $R$ equipped with two operations $+$ and $.$ such that $(R,+)$ is an abelian monoid with identity element $0$ and $(R,.)$ is also a monoid with identity element 1, along with a distributive law which respects the two operations. 
\end{def0}

An element $a$ in $R$ is called \emph{invertible} if there exists $b \in R$ such that $ab = ba =1$.We denote by $U(R)$ the set of all invertible elements in $R$. A semiring $R$ is called a \textit{semidomain} if $ab = ac$ implies $b = c$ for all $b, c \in R$ and all nonzero $a \in R$. 

As is clear from the definition, no element of a semiring has a additive inverse. For our discussions we would be considering the multiplicative operation to be commutative, unless otherwise stated.

We define a semimodule \cite{golan2013semirings}, which is based on the definition of a module over a commutative ring.

\begin{def0}
Let $R$ be a semiring. A $R$-semimodule is a commutative monoid $(M,+)$ with additive identity $\theta$ for which we have a  function $R \times M \longrightarrow M$, denoted by $(\lambda,\alpha) \rightarrow \lambda \alpha$ and called a scalar multiplication, which satisfies the following conditions for all $\lambda, \mu$ in $R$ and $\alpha,\beta$ in $M$:
\begin{enumerate}
    \item $(\lambda\mu) \alpha = \lambda(\mu \alpha);$
    \item $\lambda(\alpha + \beta) = \lambda \alpha + \lambda \beta;$
    \item $(\lambda + \mu) \alpha = \lambda \alpha + \mu \alpha;$
    \item $1 \alpha = \alpha$
    \item $\lambda \theta = \theta = 0\alpha$
\end{enumerate}
\end{def0}

\begin{eg}
For a semiring $R$, $R$ is also a $R$- semimodule over itself. Similarly, $R^{d}$ is also a $R$- semimdoule. 
\end{eg}

A nonempty subset $N$ of $M$ is said to be a subsemimodule of $M$ if it is closed under addition and scalar multiplication.

Let $S$ be a nonempty subset of a $R$-semimodule $M$. Then the intersection of all subsemimodules of $M$ containing $S$ is a subsemimodule of $M$, called the subsemimodule generated by $S$ and denoted by $RS$ \cite{YJT14}.

\[ RS = \{\sum_{i=1}^{k} \lambda_{i} \alpha_{i} | \lambda_{i} \in R, \alpha_{i} \in S, i \in \underline{k}, k \in \mathbb{N}\} \]

The expression $\sum_{i=1}^{k} \lambda_{i} \alpha_{i}$ is called a \emph{linear combination} of the elements $\alpha_{1}, \alpha_{2}, \hdots \alpha_{k}$. If $RS = M$, then $S$ is called a \emph{generating set} for $M$. We state some definitions from \cite{golan2013semirings},

\begin{def0}
Let $M$ be an $R$-semimodule. A non-empty subset $S$ of $M$ is called \emph{linearly independent} if $\alpha \not \in R(S \setminus \{ \alpha \})$ for any $\alpha$ in $S$. If $S$ is not linearly independent then it is called \emph{linearly dependent}. The set $S$ is called \emph{free} if each element in $M$ can be expressed as a linear combination of elements in $S$ in at most one way. It is clear that any free set is linearly independent.
\end{def0}
 
\begin{def0}
Let $M$ be an $R$-semimodule. A linearly independent generating set for $M$ is called a \emph{basis} for $M$ and a free generating set for $M$ is called a \emph{free basis} for $M$. A $R$-semimodule having a free basis is called a \emph{free R-semimodule}.
\end{def0}

For any commutative semiring $R$, let $\kappa(R)$ = max $\{ t \in \mathbb{N}$ $|$ the $R$-semimodule $R$ has a basis with $t$ elements $\}$.

With linear combination of elements defined over semirings, the notion of a basis is a quite natural extension. However, this as a notion is a bit subtle in this case. Firstly, for a general semimodule, bases might not have the same cardinality, i.e., two bases can have different number of elements. However, in \cite{YJT14} it is shown that if we restrict ourselves to certain semimodules over a certain class of semirings, this can be restricted, and is stated in the form of the following result, 

\begin{theorem}[Theorem 4.3 \cite{YJT14}]\label{thm:fingenbas}
For any commutative semiring $R$, the following statements are equivalent,

\begin{enumerate}
\item $\kappa(R) = 1$
\item For any $u,v \in R, 1 = u + v$ implies that either $u \in U(R)$ or $v \in U(R)$
\item Any two bases for a finitely generated free $R$-semimodule $M$ have the same cardinality. 
\end{enumerate}
\end{theorem}

We refer to a semiring with $\kappa(R) = 1$ as a \textbf{semi-linear} semiring.

An important fact mentioned in \cite{YJT14} is that the tropical semiring (both max and min) is a semi-linear semiring.

For a free, finitely generated semimodule over a semi-linear semiring, we refer to the cardinality of its basis as its \emph{dimension} or \emph{rank} of the semimodule.

We recall the definition of an ideal over a semiring \cite{YJT14}.

\begin{def0}
A nonempty subset $I$ of a semiring $R$ is said to be an ideal of $R$, if $a + b \in I$ for all $a, b \in I$ and $ ra \in I$ for all $r \in R$ and $a \in I$.
\end{def0}

\begin{def0}
An ideal $I$ of a semiring $R$ is called principal if $I = \{ra : r \in R\}$ for some $a \in R$.
\end{def0}

A semiring $R$ is called a \textit{principal ideal semidomain} if $R$ is a semidomain and all its ideals are principal.

\begin{eg}
The tropical semiring $\mathbb{T} = \{ \mathbb{R} \cup -\infty, max, + \}$ is a principal ideal semidomain. Firstly, it is easy to see that $\mathbb{T}$ is a semidomain. Also, we observe that the only ideals of this semiring are, $I_{0} = \emptyset$, the empty ideal, the single element ideal $I_{1} = \{ - \infty \}$, i.e the ideal containing the element $-\infty$ and the full semiring $\mathbb{T}$. We see that $\mathbb{T} = \langle 0 \rangle$, i.e the ideal generated by $0$ is equal to the the whole semiring. Also, $I_{1} = \{ - \infty \}  = \langle - \infty \rangle$. Hence, all ideals are principal, and therefore $\mathbb{T}$ is a principal ideal semidomain.
\end{eg}

We now look for a tropical counterpart of some classical results concerning free and finitely generated modules over principal ideal domains; the proof for Theorem \ref{thm:rank} that we provide here follows in a mostly straight forward way from the classical case \cite{mcln} and is stated here for the sake of completeness. For this we consider free and finitely generated semimodules over principal ideal semidomains.

\begin{thm0}\label{thm:rank}
Let $R$ be a principal ideal semidomain, let $F$ be a free and finitely generated $R$-semimodule and let $E$ be a subsemimodule of $F$. Then $E$ is a free and finitely generated $R$-semimodule and the rank of $E$ is at most the rank of F.
\end{thm0}

\begin{proof}
We first begin with some basic observations. We know that $R$ is a free $R$-semimodule of rank 1. Consider $E$ to be a subsemimodule of the free $R$-semimodule $R$. It is clear that $E$ is an ideal and that the ideals of $R$ coincide with the subsemimodules of $R$. In case $E$ is trivial we see that $E$ is the free $R$-semimodule of rank 0. So we consider the case that $E$ is nontrivial. Since $R$ is a principal ideal semidomain we pick $w \neq 0$ so that $E$ is generated by $w$. That is
$E = \{rw | r \in R\}$. Since we know that $R$ has $\{1\}$ as a basis, we see that the map that sends $1$ to $w$ extends to a unique semimodule homomorphism from $R$ onto $E$. Indeed, notice  $h(r \cdot 1) = r \cdot h(1) = rw$ for all $r \in R$. But the homomorphism $h$ is also one-to-one since

\begin{equation*}
\begin{split}
h(r) & = h(s) \\  rh(1) & = sh(1) \\ rw & = sw \\ & r = s     
\end{split}
\end{equation*}

Therefore, we see that $E$ is isomorphic to the free $R$-module of rank $1$. Also, we conclude that subsemimodules of the free $R$-module of rank 1 are themselves free and have either rank 0 or rank 1.

Since $F$ is a free and finitely generated semimodule, therefore Theorem \ref{thm:fingenbas} implies that all basis for $F$ have the same cardinality. Let $B$ be a basis for $F$ and $C \subseteq B$. Because $F$ is not the trivial module, we see that $B$ is not empty. Let $F_{C}$ be the subsemimodule of $F$ generated by $C$. Let $E_{C} = E \cap F_{C}$. Evidently, $C$ is a basis for $F_{C}$. To see that $E_{C}$ is free and finitely generated we will have to find a basis for it.

Suppose, for a moment, that $C$ has been chosen so that $E_{C}$ is known to be free and finitely generated and consider an element $w \in B$ with $w \not \in C$. Put $D := C \cup \{w\}$. Consider the map defined on $D$ into $R$ that sends all the elements of $C$ to $0$ and that sends $w$ to $1$. This map extends uniquely to a homomorphism of semimodules $\phi$ from $F_{D}$ onto $R$ and it is easy to check that the kernel of $\phi$ is just $F_{C}$. Contrary to the case of modules, semimodules in the context of universal algebras have salient versions of isomorphism theorems involving congruence relations and we notice that Ker $\phi$ defines a congruence relation on $F_{D}$ in this case. We invoke such isomomorphism theorem for semimodules \cite[Theorem 2.6]{pareigis13} and as $\phi$ is one-one by definition we conclude that $F_{D}/F_{C}$ is isomorphic to $R$ and that it is free of rank 1. Observe that $E_{C} = E \cap F_{C} = E \cap F_{D} \cap F_{C} = E_{D} \cap F_{C}$. Again using the isomorphism theorem, we obtain a second isomorphism theorem as follows,

\[ E_{D}/E_{C} = E_{D}/E_{D} \cap F_{C} \cong E_{D} + F_{C}/F_{C} \]

But $E_{D} +F_{C}/F_{C}$ is a subsemimodule of $F_{D}/F_{C}$. This is a free R-semimodule of rank 1. We noted earlier that every subsemimodule of a free $R$-semimodule of rank 1 must be itself a free $R$-semimodule and have rank either 0 or 1. In this way, we find that either $E_{D} = E_{C}$ (in the rank 0 case) or else $E_{D}/E_{C}$ is a free $R$-semimodule of rank 1. Let us take up this latter case. Let $X$ be a basis for $E_{C}$, which we assumed, for the moment, was free and finite. Pick $u \in E_{D}$ so that $\{u/E_{C}\}$ is a basis for $E_{D}/E_{C}$.

We claim that $X \cup \{u\}$ is a basis for $E_{D}$. Suppose $x_{0}, \hdots , x_{n-1}$ are distinct elements of $X$, $r_{0}, \hdots , r_{n} \in R$ and

\[ 0 = r_{0}x_{0} + \hdots + r_{n-1}x_{n-1} + r_{n}u \]

Also,

\[ r_{n}(u/E_{C}) = r_{n}u/E_{C} = (r_{0}x_{0} + \hdots + r_{n-1}x_{n-1} + r_{n}u)/E_{C} = 0/E_{C}\]

Since ${u/E_{C}}$ is a basis for $E_{D}/E_{C}$, we must have $r_{n} = 0$. This leads to

\[ 0 = r_{0}x_{0} + \hdots + r_{n-1}x_{n-1} \]

But now since $X$ is a basis for $E_{C}$ we see that $0 = r_{0} = \hdots = r_{n-1}$. So we find that $X \cup \{u\}$ is linearly independent. To see that $X \cup \{u\}$ generates $E_{D}$, pick $z \in E_{D}$. Since ${u/E_{C}}$ is a basis for $E_{D}/E_{C}$, pick $r \in R$ so that
$z/E_{C} = ru/E_{C}$. This means that $z-ru \in E_{C}$. But $X$ is a basis of $E_{C}$. So pick $x_{0}, \hdots , x_{n-1} \in X$ and $r_{0}, \hdots , r_{n-1} \in R$ so that $z - ru = r_{0}x_{0} + \hdots + r_{n-1}x_{n-1}$. Surely this is enough to see that $z$ is in the subsemimodule generated by $X \cup \{u\}$. So this set generates $E_{D}$ and we conclude that it must be a basis of $E$.

With this preliminary setup we try to generalize our argument for finding the free and finite basis for $E$. Notice that $E = E \cap F = E \cap F_{B}$. So $E = E_{B}$. We
start with $\emptyset \subseteq B$. We observe that $F\emptyset = E\emptyset$ is the module whose sole element is $0$. It is free of rank $0$. Next we select an element $w \in B$ and form $\emptyset \cup \{w\} = \{w\}$. We find that $E_{\{w\}}$ is free of rank $0$ or rank $1$. We select other elements until finally all the elements of $B$ have been selected. At this point we have $E_{B}$ which is free and finitely generated and its rank can be no more than the total number of elements we selected, namely $|B|$ which is the rank of $F$. 

Let $\mathcal{F} = \{f | f \text{is a function with dom} \> f \subseteq B \> \text{and range f is a basis for} \> E_{dom f} \}$. We see that $\mathcal{F}$ is partially ordered by set inclusion. We note that $\mathcal{F}$ is not empty since the empty function is a member of $\mathcal{F}$. To invoke Zorn’s lemma, let $C$ be any chain included in $\mathcal{F}$. Let $h = \bigcup C$. Evidently $f \subseteq h$ for all $f \in C$. So $h$ is an upper bound of $C$. We also conclude that $h \in \mathcal{F}$. It is also evident that $\text{dom} \> h = \bigcup \> \{ \text{dom} \> f | f \in C\}$ and that $\text{range} \> h = \bigcup \> \{ \text{range} \> f | f \in C\}$. It remains to show that range $h$ is a basis for $E_{\text{dom h}}$. To see that range $h$ is a generating set, let $z$
be an arbitrary element of $E_{\text{dom h}} = E \cap F_{\text{dom h}}$. Hence $z$ must be generated by some finitely many elements belong in $\text{dom h}$. This means there are finitely many functions $f_{0}, \hdots , f_{n-1} \in C$ so that $z$ is generated by finitely many elements of $\text{dom} f_{0} \cup \hdots \cup \text{dom} f_{n-1}$. But $\text{dom} f_{0} \cup \hdots \cup \text{dom} f_{n-1}$, under rearrangement, forms a chain under inclusion. So $z \in F_{\text{dom} f_{l}}$
for some $l < n$. Hence $z \in E_{\text{dom} f_{l}}$. But $\text{range} f_{l}$ is a basis for $E_{\text{dom} f_{l}}$. Because $\text{range} f_{l} \subseteq \text{range} h$ we find that range $h$
has enough elements to generate $z$. Since $z$ was an arbitrary element of $E_{\text{dom} h}$ we conclude that range $h$ generates $E_{\text{dom} h}$. It remains to show that range $h$ is linearly independent. But range $h$ is the union of the chain $\{\text{range} f | f \in C\}$. We recall that
the union of any chain of linearly independent sets must also be linearly independent. This implies that $h$ belongs to $\mathcal{F}$. By Zorn's lemma, let $g$ be a maximal element of $\mathcal{F}$.
We are done if $\text{dom} g = B$, since then $E = E \cap F = E \cap F_{B} = E_{B} = E_{\text{dom} g}$. In which case, range $g$ would be a basis for $E$ and rank $E$ = $|\text{range} \> g| \leq | \text{dom} \> g| = |B| = \text{rank} \> F$.

Consider the possibility that dom $g$ is a proper subset of $B$. Put $C$ = dom $g$ and put $X$ = range $g$. Let $w \in B$ with $w \not\in$ dom $g$. Put $D = C \cup \{w\}$. As we have seen above, either $E_{D} = E_{C}$ or $X \cup \{u\}$ is a basis for $E_{D}$, for some appropriately chosen $u$. We can now extend $g$ to a function $g'$ by letting $g'(w)$ be any element of range $g$ in the case when $E_{D} = E_{C}$ and by letting $g'(w) = u$ in the alternative case. In this way, $g' \in \mathcal{F}$, contradicting the maximality of $g$. So we negate this possibility.

This completes the proof.

\end{proof}

We refer to a \textbf{semi-linear} semiring which is also a principal ideal semidomain as a \textbf{linear} semiring.

\begin{eg}
The tropical semiring $\mathbb{T} = \{ \mathbb{R} \cup \{ -\infty\}, \> \text{max}, \> + \}$ is a linear semiring.
\end{eg}

We now define the notion of a projective space over a semiring,

\begin{def0}\label{def:proj_space_over_semimodules}
Given a semimodule $M$ over a semiring $R$, the \textbf{projective space over a semimodule} is the quotient space of the semimodule (omitting the additive identity $0$) under scalar multiplication, omitting multiplication by the scalar additive identity $0$.
\[ R(M) = (M \setminus 0) / (R \setminus 0)    \]
\end{def0}

\begin{eg}
The \emph{tropical projective space} of dimension $d-1$ is \cite{joswig2014essentials},
\begin{equation*}
\begin{split}
 {\mathbb{T}\mathbb{P}}^{d-1} = (\mathbb{T}^{d} \setminus \mathbf{\infty}) \> / \> \mathbb{R} \cdot 1 \\
\> \text{where}, \> \infty = (\infty, \hdots , \infty)
\end{split}
\end{equation*} 
\end{eg}

As elaborated in \cite{A16}, a vector space $V$ can be attached with the notion of a corresponding projective space $\overline{V}$, in which its elements are the subspaces $U \subset V$, and each subspace $U$ has a projective dimension dim$_{p} U$  = dim $U$ - 1. Similarly, we define $\overline{V}$ to be the \textbf{associated projective space} corresponding to a free, finitely generated semimodule $V$ over a linear semiring $R$, which has its elements as subsemimodules $U$ of $V$ and the \textbf{dimension} of $U$ = \#(basis($U$)) - 1. This evidently coincides with the notion of a projective space in Definition \ref{def:proj_space_over_semimodules}, as defined earlier, and this is illustrated in the following example. 

\begin{eg}
Let us consider the semimodule $V = \mathbb{T}^{3}, \> \text{then} \> \mathbb{T}\mathbb{P}^{2} = (\mathbb{T}^{3} \setminus \{-\infty\}) / \mathbb{R} \cdot 1$. In this case we also see that for $V = \mathbb{T}^{3}$  the associated projective space consists of subsmsemimodules of $ \mathbb{T}^{3}$, like  $\mathbb{T}^{2}$ of dimension two and  $\mathbb{T}^{1}$ of dimension one. These subsemimodules correspond to the points in $\mathbb{T}\mathbb{P}^{2}$, like $\mathbb{T}^{2}$ corresponds to $(0,0,\infty) / \mathbb{R} \cdot 1$ and $\mathbb{T}^{1}$ corresponds to $(0,\infty,\infty) / \mathbb{R} \cdot 1$. We refer the reader to Figure \ref{fig:tropical_projective_plane} for a complete description of $\mathbb{T}\mathbb{P}^{2}$. We realize that the scalar multiplication in the subsemimodules in  the definition of the associated projective space is compensated by the action of  $\mathbb{R} \cdot 1$ in Definition \ref{def:proj_space_over_semimodules} and hence these two definitions coincide.
\end{eg}

Therefore, the "points" of $V$ are subspaces of projective dimension $0$ and the "lines" are subspaces of projective dimension 1. Thus the lines of $V$ become the "points" of $\overline{V}$ and the planes of $V$ become the "lines" of $\overline{V}$ \cite{A16}. 

We now take a close look at the behaviour of points and lines over a semiring . A line over a semiring $R$ is defined as a one-dimensional subsemimodule of a semimodule $V$ defined over $R$. 

In the case of the tropical plane, a tropical line is defined as a hypersurface by the linear polynomial 

\[ L \equiv p(x,y) = a \otimes x \oplus b \otimes y \oplus c,   \>\> a,b,c \in R  \] 
 
where $L$ is a one dimensional subsemimodule of $V= \mathbb{R}^{2}$ over $R = \mathbb{R}$, which is equal to the corner locus of three half rays emanating from the point $(c-a,c-b)$, which we refer as the \emph{vertex} of the tropical line, in the primitive directions of $(-1,0), (0,-1), (1,1)$ \cite{maclagan2009introduction}, as described in Figure \ref{fig:tropical_line}.

\begin{remark}
At this juncture we would also like to point the reader to \cite[Section 5.2]{joswig2014essentials} and \cite[Remark 5.2.2]{maclagan2009introduction} which relate equivalence of tropical cones and semimodules and equivalence of tropically convex sets and subsemimodules in $\mathbb{T}^{d} \> / \mathbb{R} \cdot 1$, respectively. For readers familiar with tropical convexity, our discussion regarding semimodules over semirings can also be visualized in the context of tropical convexity, at least in the case when the underlying semiring is the tropical semiring. We elaborate on the connections between our results and tropical convexity in Section \ref{sect:Section3}.
\end{remark}

\begin{figure}[H]
    \centering
    \includegraphics[scale=0.9]{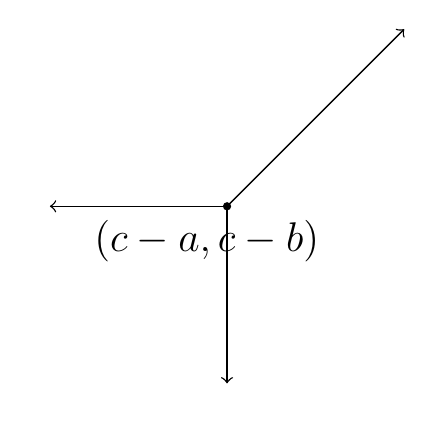}
    \caption{A tropical line}
    \label{fig:tropical_line}
\end{figure}

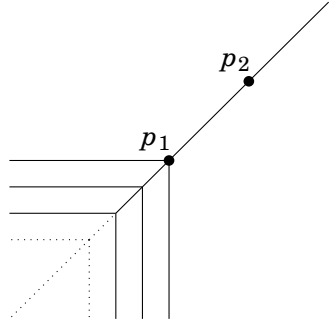
\begin{figure}
\centering
\begin{tikzpicture}[scale=0.7]
\draw[][] (0,0) -- (3,3);
\draw[][] (0,0) -- (0,-3);
\draw[][] (0,0) -- (-3,0);
\draw[][] (-0.5,-0.5) -- (-3,-0.5);
\draw[][] (-0.5,-0.5) -- (-0.5,-3);
\draw[][] (-1,-1) -- (-3,-1);
\draw[][] (-1,-1) -- (-1,-3);
\draw[dotted][] (-1.5,-1.5) -- (-3,-1.5);
\draw[dotted][] (-1.5,-1.5) -- (-1.5,-3);
\draw[][] (-1,-1) -- (0,0);
\draw[dotted][] (-1.5,-1.5) -- (-1,-1);
\draw[dotted][] (-1.5,-1.5) -- (-3,-3);
\fill[black] (0,0) circle (.1cm) node[align=left,   above]{$p_{1}\quad$};
\fill[black] (1.5,1.5) circle (.1cm) node[align=left,   above]{$p_{2}\quad$};
\end{tikzpicture}
\caption{Two points defining infinite number of tropical lines}
\label{fig:infinite_lines}
\end{figure}

A feature of such a geometric definition of a tropical line is that two tropical lines can intersect at more than one point; they can intersect over a a common half ray. Two tropical lines which have a unique intersection are said to be in \emph{general position}. Also, the tropical plane allows projective duality \cite{brandt2018incidence}, hence considering the projective dual, two points in the tropical plane can define infinite number of lines passing through them (cf. Figure \ref{fig:infinite_lines}). This leads us to the following definition \cite{brandt2018incidence},

\begin{def0}
Two points are said to be \emph{coaxial} if they lie on the same axis of a tropical line containing them.
\end{def0}

Two tropical lines are said to be coaxial if there vertices are coaxial. As is evident, this definition of coaxial lines is specific to the tropical semiring.

We also recall the definition of a \emph{stable tropical line} from \cite{tewari2020pointline},

\begin{def0}
Consider $(L, p_{1}, \hdots, p_{n}), (n \geq 2)$ where $L$ is a tropical line with the points $(p_{1}, \hdots, p_{n})$ on the line $L$, then $(L, p_{1}, \hdots, p_{n})$, is called \textbf{stable} if
\begin{enumerate}
\item either $L$ is the unique line passing through the $p_{i}$'s,  or
\item one of the points  $p_{1}, \hdots, p_{n}$ is the vertex of $L$.
\end{enumerate}
\end{def0}

\begin{figure}
    \centering
    \includegraphics[scale=0.9]{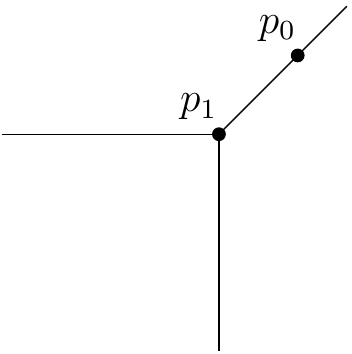}
    \caption{A stable tropical line $(L,p_{0},p_{1})$}
    \label{fig:stable_line}
\end{figure}

We would now consider a line over an arbitrary semiring. We first recall the axioms which govern incidence of points and lines classically,

\begin{axiom}[Axiom 1 \cite{A16}]
Given two points $P$ and $Q$, there exists a unique line $l$ such that $P$ lies on $l$ and $Q$ lies on $l$. 
\end{axiom}

The unique line containing $P$ and $Q$ is referred as $l = P + Q$.

Tropically, we define the following axiom which underlines point line incidence over semirings,

\begin{axiom}
Given two points $P$ and $Q$, there either exists a unique line $l$ such that $P$ lies on $l$ and $Q$ lies on $l$ or there exist infinite number of lines such that $P$ and $Q$ both lie on them. 
\end{axiom}

We can now provide an axiomatic description of \emph{coaxiality}, similar to the classical axiomatic description of \emph{collinearity}.

\begin{def0}
Given two points $P$ and $Q$, they are said to be \emph{coaxial} if there exists an infinite number of lines containing both $P$ and $Q$.
\end{def0}

We also refer to the lines containing two coaxial points $P$ and $Q$, as coaxial lines.
For two coaxial points, among the infinite number of lines containing them, coaxiality defines an equivalence relation on the infinite coaxial lines.


Given two points $P$ and $Q$, $l = P \oplus Q$ denotes either the unique line containing them or the class of infinite coaxial lines containing $P$ and $Q$. 








We now recall definitions of two classical maps regarding vector spaces over fields and their associated projective space. Let $V$ and $V'$ be two vector spaces over the fields $k$ and $k'$, and let $\mu$ be an isomorphism between $k$ and $k'$.

\begin{def0}[Definition 2.10, \cite{A16}]
A map $\lambda : V \rightarrow V'$ is called a \textbf{semi-linear} with respect to the isomorphism $\mu$ if \begin{enumerate}
    \item $\lambda(X + Y) = \lambda(X) + \lambda(Y), $
    \item $\lambda(\alpha X) = \alpha^{\mu} \lambda(X)$ for all $X,Y \in V$ and for all $\alpha \in k.$  
\end{enumerate} 
\end{def0}

We refer to the group of semilinear transformations as $\pi(V,V')$. In the case of $V=V'$, this is a group of automorphisms which contains a normal subgroup isomorphic to $k^{*}$. We refer to the quotient of $\pi(V)$ with this normal subgroup as $P\pi(V)$.

\begin{def0}[Definition 2.11, \cite{A16}]
A map $\sigma : \overline{V} \rightarrow \overline{V'}$ of the elements of a projective space $\overline{V}$ onto the elements of a projective space $\overline{V'}$ is called a \textbf{collineation} if 
\begin{enumerate}
    \item dim $V$  = dim $V'$,
    \item $\sigma$ is one to one and onto,
    \item $U_{1} \subset U_{2} \implies \sigma U_{1} \subset \sigma U_{2}$.
\end{enumerate}
\end{def0}

We now define tropical analogues of these two maps between two semimodules $V$ and $V'$ defined over two linear semirings $R$ and $R'$, which are isomorphic with respect to an isomorphism $\mu$.

\begin{def0}
A map $\lambda : V \rightarrow V'$ is called \textbf{semiring semi-linear} with respect to the isomorphism $\mu$ if 
\begin{enumerate}
    \item $\lambda(X + Y) = \lambda(X) + \lambda(Y), $
    \item $\lambda(\alpha X) = \alpha^{\mu} \lambda(X)$ for all $X,Y \in V$ and for all $\alpha \in R.$  
\end{enumerate}
\end{def0}

\begin{def0}
A map $\sigma : \overline{V} \rightarrow \overline{V'}$ of the elements of a associated projective space $\overline{V}$ onto the elements of a associated tropical space $\overline{V'}$ is called a \textbf{semiring collineation} if 
\begin{enumerate}
    \item dim $V$  = dim $V'$,
    \item $\sigma$ is one to one and onto,
    \item $U_{1} \subset U_{2} \implies \sigma U_{1} \subset \sigma U_{2}$.
\end{enumerate}
\end{def0}

\section{Main Theorem}\label{sect:Section3}



We first state the classical Fundamental Theorem of Projective Geometry as stated in \cite{A16},

\begin{thm0}[Theorem 2.26 \cite{A16}]\label{thm:classfund}
Let $V$ and $V'$ be left vector spaces of equal dimensions $n \geq 3$ over fields $k$ and $k'$ respectively, $\overline{V}$ and $\overline{V'}$ be the corresponding projective spaces. Let $\sigma$ be a one-to-one (onto) correspondence which has the following property: Whenever three distinct "points" $L_{1},  L_{2}$ and $L_{3}$ are collinear : $L_{1} \subset L_{2} + L_{3}$, then there images are collinear: $\sigma L_{1} \subset \sigma L_{2} + \sigma L_{3}$. There exists an isomorphism $\mu$ of $k$ onto $k'$ and a semi-linear map $\lambda$ of $V$ onto $V'$ (with respect to $\mu$) such that the collineation which $\lambda$ induces on $\overline{V}$ agrees with $\sigma$ on the points of $\overline{V}$. If $\lambda_{1}$ is another semi-linear map with respect to an isomorphism $\mu_{1}$ of $k$ onto $k'$ which also induces this collineation, then $\lambda_{1}(X) = \lambda(\alpha X)$ for some fixed $\alpha \neq 0$ of $k$ and the isomorphism $\mu_{1}$ is given by $x^{\mu_{1}} = {(\alpha x \alpha^{-1})}^{\mu}$. For any $\alpha \neq 0$ the map $\lambda(\alpha X)$ will be semi-linear and induce the same collineation as $\lambda$. 
\end{thm0}

With the help of definitions discussed earlier, we now state a version of Fundamental Theorem of Projective Geometry over semirings as follows,

\begin{thm0}[Fundamental Theorem of Projective Geometry over semirings]\label{thm:tropicalfund}
Let $V$ and $V'$ be free, finitely generated semimodules of equal dimensions $n \geq 3$ over linear semirings $R$ and $R'$ respectively. Let $\overline{V}$ and $\overline{V'}$ be the corresponding associated projective spaces. Let $\sigma$ be a one-to-one (onto) correspondence of the "points" of $\overline{V}$ and "points" of $\overline{V'}$ which has the following property: whenever three distinct "points" $L_{1},  L_{2}$ and $L_{3}$ are coaxial : $L_{1} \subset L_{2} \oplus L_{3}$, then there images are coaxial: $\sigma L_{1} \subset \sigma L_{2} \oplus \sigma L_{3}$. There exists an isomorphism $\mu$ of $R$ onto $R'$ and a semiring semi-linear map $\lambda$ of $V$ onto $V'$ (with respect to $\mu$) such that the semiring collineation which $\lambda$ induces on  $\overline{V}$ agrees with $\sigma$ on the points of $\overline{V}$. For any $\alpha \neq 0$ the map $\lambda(\alpha X)$ is semiring semi-linear and induces the same semiring collineation as $\lambda$.  
\end{thm0}

\begin{proof}
We begin the proof similar to the setup used in the classical case in \cite{A16} and \cite{putman}. We divide the proof in various small results,

\begin{lemma}\label{lem:lem1}
Let $v_{1}, \hdots ,v_{p} \in V$ and $\langle v_{i} \rangle = L_{i}$ (for all $ 1 \leq i \leq p) \in \overline{V}$. Also, let $v'_{1}, \hdots ,v'_{p} \in V'$ and $\langle v'_{i} \rangle = L'_{i}$ (for all $ 1 \leq i \leq p) \in \overline{V'}$. Then 

\[  \sigma( \langle v_{1}, \hdots , v_{p} \rangle)  = \langle v'_{1}, \hdots ,v'_{p} \rangle \]
\end{lemma}

\begin{proof}
The subsemimodule $\langle v_{1}, \hdots , v_{p} \rangle$, by definition is the minimal subsemimodule of V containing all $ \langle v_{i} \rangle = L_{i}$. Since $\sigma$ is a map which preserves dimension and is one to one and onto, hence $\sigma ( \langle v_{1}, \hdots , v_{p} \rangle )$ is the minimal subsemimodule of $V'$ containing each $\sigma( \langle v_{i} \rangle ) = \langle v'_{i} \rangle =  L'_{i}$. Hence, the claim.
\end{proof}

Let $\{ v_{1}, \hdots , v_{n} \}$ be a basis for $V$ and we want to construct a basis $\{ v'_{1}, \hdots , v'_{n} \}$ for $V'$. Let $v'_{1} \in V'$, such that $\sigma ( \langle v_{1} \rangle ) = \langle v'_{1} \rangle.$ This choice of $v'_{1}$ will be our only arbitrary choice; everything else will be determined by it. We now construct $\{ v'_{2}, \hdots , v'_{n} \}$.

\begin{lemma}
For $2 \leq i \leq n$, there exists a unique $v'_{i} \in V'$ such that 

\[ \sigma( \langle v_{i} \rangle ) = \langle v'_{i} \rangle \quad  \text{and} \quad \sigma( \langle v_{1} + v_{i} \rangle )  = \langle v'_{1} + v'_{i} \rangle \]

Moreover, the set $\{ v'_{1}, \hdots , v'_{n} \}$ is a basis for $V'$.
\end{lemma}

\begin{proof}
Pick $w_{i}$s $\in V' $, such that $\sigma( \langle v_{i} \rangle ) = \langle w_{i} \rangle$. Using Lemma \ref{lem:lem1}, we have 

\[ \sigma ( \langle v_{1} + v_{i} \rangle ) \subset \sigma( \langle v_{1}, v_{i} \rangle ) = \langle v'_{1}, w_{i} \rangle    \] 

Since $\sigma( \langle v_{1} + v_{i} \rangle ) \neq  \langle w_{i} \rangle$, therefore it follows that there exists a unique $\gamma_{i} \in R'$ such that 

\[ \sigma ( \langle v_{1} + v_{i} \rangle ) =  \langle v'_{1} + \gamma_{i} w_{i} \rangle  \]

therefore the desired vector $v'_{i} = \gamma_{i} w_{i}$. To deduce that $\{ v'_{1}, \hdots , v'_{n} \}$ forms a basis, we use Lemma \ref{lem:lem1} and the fact that $\sigma$ is one to one and onto.
\end{proof}

We now try to construct the isomorphism map $\mu : R \longrightarrow R'$,

\begin{lemma}
For $2 \leq i \leq n$, there exists a unique set map $\mu : R \longrightarrow R'$ such that

\[ \sigma ( \langle v_{1} + cv_{i} \rangle ) =  \langle v'_{1} + \mu_{i}(c) \> v'_{i} \rangle \]

($c \in R$)
\end{lemma}

\begin{proof}
We define $\mu_{i}$ as follows, consider $c \in R$. We apply Lemma \ref{lem:lem1} to get 

\[ \sigma ( \langle v_{1} + c v_{i} \rangle )  \subset \sigma ( \langle v_{1}, v_{i} \rangle ) = \langle v'_{1}, v'_{i} \rangle  \]

Since $\sigma ( \langle v_{1} + c v_{i} \rangle ) \neq \langle v'_{i} \rangle $, therefore there exists a unique $\mu_{i} (c) \in R'$, such that 

\[ \sigma ( \langle v_{1} + c v_{i} \rangle )  = \langle v'_{1} + \mu_{i}(c) \> v'_{i} \rangle \]

The uniqueness of $\mu_{i}$ follows from the fact that $\mu_{i} (0_{R}) = 0_{R'}$ and $\mu_{i} (1_{R}) = 1_{R'}$, where $0_{R}$ and $0_{R'}$ represents the additive identities of the semirings $R$ and $R'$ respectively. Similarly, $1_{R}$ and $1_{R'}$ represents the multiplicative identities of the semirings $R$ and $R'$ respectively.

\end{proof}

\begin{lemma}
For distinct $2 \leq i,j \leq n$, we have $\mu_{i} = \mu_{j}$. 
\end{lemma}

\begin{proof}
Consider a non-zero $c \in R$. We have 

\[ \langle v_{i} - v_{j} \rangle  \subset \langle v_{i}, v_{j} \rangle  \quad \text{and} \quad \langle v_{i} - v_{j}  \rangle \subset  \langle v_{1} + cv_{i}, v_{1} + cv_{j} \rangle  \] 

Applying Lemma \ref{lem:lem1}, we get 

\[ \sigma (\langle v_{i} - v_{j} \rangle ) \subset \langle v'_{i}, v'_{j} \rangle  \quad \text{and} \quad   \sigma ( \langle v_{i} - v_{j}  \rangle ) \subset  \langle v'_{1} + \mu_{i} (c) \> v'_{i}, v'_{1} + \mu_{j} (c) \> v'_{j} \rangle  \]

We have 

\[ \langle v'_{i}, v'_{j} \rangle \cap  \langle v'_{1} + \mu_{i} (c) \> v'_{i}, v'_{1} + \mu_{j} (c) \> v'_{j} \rangle =  \langle \mu_{i} (c) \> v'_{i} - \mu_{j} (c) \> v'_{j} \rangle  \]

Therefore, 

\[ \sigma ( \langle v_{i} - v_{j} \rangle ) = \langle \mu_{i} (c) v'_{i} - \mu_{j} (c) v'_{j} \rangle \]

since the left hand side is independent of $c$ in this equation, therefore the right hand side should also be independent of $c$. Hence,

\[ \langle v'_{i} - v'_{j} \rangle  = \langle \mu_{i} (1_{R}) v'_{i} - \mu_{j} (1_{R}) v'_{j} \rangle
 = \langle \mu_{i} (c) v'_{i} - \mu_{j} (c) v'_{j} \rangle \]
 
Therefore,

\[ \mu_{i} (c)  = \mu_{j} (c) , \quad \forall c \in R \]
\end{proof}

\begin{lemma}\label{lem:lem5}
For $c_{1}, \hdots , c_{n} \in V$, we have 

\[ \sigma ( \langle v_{1} + c_{2} v_{2} + \hdots + c_{n} v_{n} \rangle )  = \langle v'_{1} + \mu (c_{2}) \> v'_{2} + \hdots + \mu (c_{n}) \> v'_{n} \rangle \]
\end{lemma}

\begin{proof}
We assume,

\[ \sigma ( \langle v_{1} + \hdots + c_{p} v_{p}  \rangle ) = \langle v'_{1} + \mu (c_{2}) \> v'_{2} + \hdots + \mu (c_{p}) \> v'_{p} \rangle \]

for all $2 \leq p \leq n$, by using induction on $p$. The base case is $p = 2$ and it holds true as that is the defining property of $\mu$, so we consider $2 < p \leq n$. Applying Lemma \ref{lem:lem1} and the induction hypothesis we get,

\[ \sigma ( \langle v_{1} + \hdots + c_{p} v_{p}  \rangle ) \subset \sigma ( \langle v_{1} + \hdots + c_{p-1} v_{p-1}, v_{p}  \rangle ) = \langle v'_{1} + \mu (c_{2}) \> v'_{2} + \hdots + \mu (c_{p-1}) \> v'_{p-1}, v'_{p} \rangle \] 

Moreover, $\sigma ( \langle v_{1} + c_{2} v_{2} +  \hdots c_{p} v_{p} \rangle )$ is not $ \langle v'_{p} \rangle $, so we deduce that there exists some $d \in R'$, such that 

\[ \sigma ( \langle v_{1} + c_{2} v_{2} + \hdots + c_{p} v_{p} \rangle ) = \langle v'_{1} + \mu(c_{2}) \> v'_{2} + \hdots + \mu(c_{p-1}) \> v'_{p-1} + dv'_{p} \rangle   \]

Applying Lemma \ref{lem:lem1} and defining property of $\mu$, we see 

\[ \sigma (  \langle v_{1} + c_{2} v_{2} +  \hdots + c_{p} v_{p} \rangle  )  \subset \sigma (  \langle v_{1}  + c_{p}v_{p}, v_{2}, \hdots , v_{p-1} \rangle )  = \langle v'_{1} + \mu(c_{p}) \> v'_{p}, v'_{2}, \hdots, v'_{p-1} \rangle   \]

Therefore, $d = \mu(c_{p})$.
\end{proof}

\begin{lemma}\label{lem:lem6}
For $c_{2}, \hdots , c_{n} \in R$, we have 

\[ \sigma ( \langle c_{2}v_{2} + \hdots + c_{n}v_{n} \rangle ) = \langle \mu(c_{2}) \> v'_{2} + \hdots + \mu(c_{n}) \> v'_{n} \rangle     \]

\end{lemma}

\begin{proof}

By Lemma \ref{lem:lem1}, we have 

\[ \sigma ( \langle c_{2}v_{2} +  \hdots + c_{n}v_{n} \rangle ) \subset \sigma ( \langle v_{2}, \hdots, v_{n} \rangle ) = \langle v'_{2}, \hdots, v'_{n} \rangle  \]

also with Lemma \ref{lem:lem1} and Lemma \ref{lem:lem5}, we have 

\[ \sigma ( \langle c_{2}v_{2} +  \hdots + c_{n}v_{n} \rangle ) \subset  \sigma ( \langle v_{1},  v_{1} + c_{2}v_{2}, + \hdots + c_{n}v_{n} \rangle ) = \langle v'_{1}, v'_{1} + \mu(c_{2}) \> v'_{2} + \hdots + \mu(c_{n}) \> v'_{n} \rangle \]

and we get,

\[ \sigma ( \langle c_{2}v_{2} +  \hdots + c_{n}v_{n} \rangle ) = \langle \mu(c_{2}) \> v'_{2} + \hdots + \mu(c_{n}) \> v'_{n} \rangle  \]
\end{proof}

We now show that the map $\mu$ is an isomorphism of semirings.

\begin{lemma}\label{lem:lem7}
For $c,d \in R$ we have $\mu(c+d) = \mu(c) + \mu(d)$.
\end{lemma}

\begin{proof}
By Lemma \ref{lem:lem5}, we have 

\[ \sigma ( \langle v_{1} + (c+d) v_{2} + v_{3} \rangle )  = \langle v'_{1} + \mu(c+d) \> v'_{2} + v'_{3} \rangle      \] 

By combining Lemma \ref{lem:lem1} with Lemma \ref{lem:lem5} and Lemma \ref{lem:lem6}, we have

\[ \sigma ( \langle v_{1} + (c+d) v_{2} + v_{3} \rangle ) \subset \sigma ( \langle v_{1} +cv_{2}, dv_{2} + v_{3} \rangle ) = \langle v'_{1} + \mu(c) \> v'_{2}, \mu(d) \> v'_{2} + v'_{3} \rangle  \] 

combining these equations we get

\[ \langle v'_{1} + \mu(c+d) v'_{2} + v'_{3} \rangle  \subset \langle v'_{1} + \mu(c) \> v'_{2}, \mu(d) \> v'_{2} + v'_{3} \rangle   \]

this can only hold true if 

\[ \mu(c+d) = \mu(c) + \mu(d) \]
\end{proof}

\begin{lemma}\label{lem:lem8}
For $c,d \in R$, we have 

\[  \mu(cd)  = \mu(c) \cdot \mu(d) \]

\end{lemma}

\begin{proof}
By Lemma \ref{lem:lem5}, we have

\[ \sigma ( \langle v_{1} + cdv_{2} + cv_{3} \rangle ) = \langle v'_{1} +\mu(cd) \> v'_{2} + \mu(c) \> v'_{3} \rangle \] 

combining Lemma \ref{lem:lem1} and Lemma \ref{lem:lem6}, we get

\[ \sigma ( \langle v_{1} + cdv_{2} + cv_{3} \rangle ) \subset \sigma (\langle v_{1}, dv_{2} + v_{3} \rangle ) = \langle v'_{1}, \mu(d) \> v'_{2} + v'_{3} \rangle \]

combining these two equations we get 

\[ \langle v'_{1} +\mu(cd) \> v'_{2} + \mu(c) \> v'_{3} \rangle \subset \langle v'_{1}, \mu(d) \> v'_{2} + v'_{3} \rangle  \] 

the only way the this holds is iff 

\[ \mu(cd) = \mu(c) \cdot \mu(d) \]
\end{proof}

\begin{lemma}
The map $\mu : R \longrightarrow R'$ is an isomorphism of semirings.
\end{lemma}

\begin{proof}
We know that $\mu(0_{R}) = \mu(0_{R'})$ and $\mu(1_{R}) = \mu(1_{R'})$, therefore $\mu$ is a one-one onto map. Also, by Lemma \ref{lem:lem7} and Lemma \ref{lem:lem8} we conclude that $\mu$ is a semiring homomorphism. Hence, $\mu$ is an isomorphism of semirings.
\end{proof}

We now have constructed our basis $\{ v'_{1}, \hdots , v'_{n} \}$ for $V'$ and the isomorphism $\mu : R \longrightarrow R'$ , so we can define semiring semi-linear map $\lambda : V \longrightarrow V'$ via the formula 

\[ \lambda (c_{1} v_{1} + \hdots + c_{n} v_{n}) = \mu(c_{1}) \> v'_{1} + \hdots + \mu(c_{n}) \> v'_{n}      \quad  (c_{1}, \hdots, c_{n} \in R) \]

\begin{lemma}
The semiring semi-linear map $\lambda : V \longrightarrow V'$ induces the semiring collineation $\sigma$.
\end{lemma}

\begin{proof}
Consider a subspace $U$ of $V$. We can write $U = \langle u_{1}, \hdots, u_{p} \rangle$, where each $u_{i}$ is either of the form $v_{1} + c_{2}v_{2} + \hdots + c_{n}v_{n}$ or of the from $c_{2}v_{2} + \hdots + c_{n}v_{n}$ for some $c_{2}, \hdots, c_{n} \in R$. Combining Lemma \ref{lem:lem1}, Lemma \ref{lem:lem5} and Lemma \ref{lem:lem6}, we see that $\sigma(U) = \langle \lambda(u_{1}), \hdots, \lambda(u_{p}) \rangle $, as desired.
\end{proof}

With all the above ten lemmas we complete the proof of the main theorem.
\end{proof}

The statement of Theorem \ref{thm:classfund} is very general and mostly in literature it is stated in the case when $k = k'$ and $V = V'= k^{n}$ and the result is stated in the following form,

\begin{thm0}\label{thm:alternatefund}
If $k$ is a field and $n \geq 3$, \text{Aut}$(\overline{V}) = P\pi(V) $
\end{thm0}

Similarly, when we consider in the case of semirings that $R = R'$ and $V = V'= R^{n}$, then with Theorem \ref{thm:tropicalfund}, we have

\begin{thm0}\label{thm:alternatefund_trop}
If $R$ is a semiring and $n \geq 3$, \text{Aut}$(\overline{V}) = P\pi'(V)$
\end{thm0}

where $\pi'(V)$ refers to the group of semiring semi-linear automorphisms on $V$ and $P\pi'(V)$ is the quotient of $\pi'(V)$ with the group of automorphisms of $R$.

We now highlight some connections between our results on semimodules and \emph{tropical convexity}, wherein the underlying semiring is the tropical semiring $\mathbb{T} = \{ \mathbb{R} \cup \{- \infty \}, \text{max}, + \}$. We recall some basic definitions concerning tropical convexity \cite{maclagan2009introduction},


\begin{def0}
A subset $S$ of $\mathbb{T}^{n}$ is \emph{tropically convex} if $x,y \in S$ and $a,b \in T$ implies $a \odot x \oplus b \odot y \in S$. The \emph{tropical convex hull} of a given subset $V \subset \mathbb{T}^{n}$ is the smallest tropically convex subset of $\mathbb{T}^{n}$ that contains $V$.  
\end{def0}

A \emph{tropical polytope} is the tropical convex hull of a finite subset $V$ in the tropical projective space. 

In \cite{matrix2018tropical}, the focus is studying on tropical matrix groups, especially over the finitary tropical semiring $\mathbb{F}\mathbb{T} = ( \mathbb{R}, max, +) $ and they consider tropical polytopes over this semiring and prove a similar result as Theorem \ref{thm:tropicalfund} in the case of this semiring,

\begin{thm0}[Theorem 5.4, \cite{matrix2018tropical}]\label{thm:finite_trop}
Every automorphism of a projective n-polytope in $\mathbb{F}\mathbb{T}^{n}$
\begin{itemize}
    \item extends to an automorphism of $\mathbb{F}\mathbb{T}^{n}$; and 
    \item is a (classical) affine linear map.
\end{itemize}
\end{thm0}

Since, all projective tropical polytopes are free and finitely generated, hence the statement of  Theorem \ref{thm:finite_trop} can be seen as a special case of the general statement in Theorem \ref{thm:tropicalfund} when the underlying linear semiring is $\mathbb{F}\mathbb{T}$.

We now look at case when the semirings $R = R'$ and $V = V'$ in Theorem \ref{thm:tropicalfund}, and one such example is the following,

\begin{eg}
We consider the case $R = R' = \mathbb{T} = \{ \mathbb{R} \cup \{ \infty\}, \text{max}, + \}$ and $V = V' = \mathbb{T}^{2}$. In this case, we obtain $\overline{V} = \overline{V'} = \mathbb{T}\mathbb{P}^{2} = (\mathbb{T}^{3} \setminus \{-\infty\}) / \mathbb{R} \cdot 1$, which is referred as the \emph{tropical projective plane} and is an example of a \emph{tropical toric variety} \cite{maclagan2018tropical}\cite{joswig2014essentials}.

We note that the map -log is a homeomorphism from ${\mathbb{R}}^{d}_{\geq 0}$ to $\mathbb{T}^{d}$ when taken componentwise,

\begin{equation*}
\begin{split}
\text{-log}: \mathbb{R}^{d}& \longrightarrow \mathbb{T}^{d} \\
    (v_{1}, \hdots , v_{n})& \longrightarrow  (\text{-log}(v_{1}), \hdots , \text{-log}(v_{n}))
\end{split}
\end{equation*}           

under this map the standard basis vector $e_{k} = (0, \hdots , 0,  1, 0,  \hdots, 0)$ is mapped to 
 
\[ e^{\text{trop}}_{k} = ( \infty, \hdots , \infty, 0, \infty, \hdots , \infty) \]

which provides us the description in Figure \ref{fig:tropical_projective_plane} of $\mathbb{T}\mathbb{P}^{2}$ \cite{joswig2014essentials}.

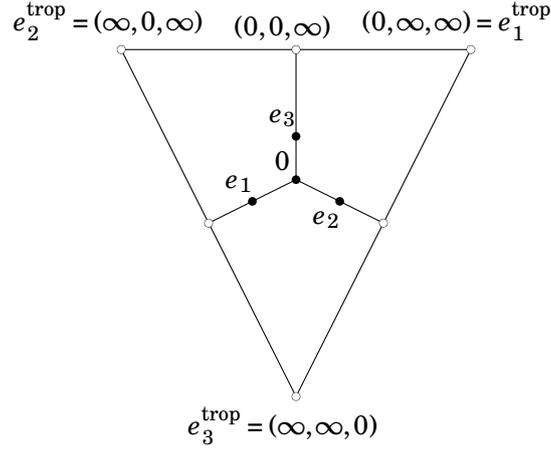
\begin{figure}[H]
\centering
\begin{tikzpicture}[scale=1.15]
\draw[][] (-2,0) -- (2,0);
\draw[][] (0,-4) -- (-2,0);
\draw[][] (2,0) -- (0,-4);
\draw[][] (0,0) -- (0,-1.5);
\draw[][] (0,-1.5) -- (1,-2);
\draw[][] (-1,-2) -- (0,-1.5);
\fill[black] (0,-1) circle (.05cm) node[align=left,   above]{$e_{3}\quad$};
\fill[black] (0,-1.5) circle (.05cm) node[align=right,   above]{$0\quad$};
\fill[black] (-0.5,-1.75) circle (.05cm) node[align=left,   above]{$e_{1}\quad$};
\fill[black] (0.5,-1.75) circle (.05cm) node[align=left,   below]{$e_{2}\quad$};
\fill[black] (0,0) circle (.05cm) node[black,align=left,   above]{$(0,0, \infty)\quad$};
\fill[white] (0,0) circle (.045cm) node[black,align=left,   above]{};
\fill[black] (-2,0) circle (.05cm) node[black,align=left,   above]{$e^{\text{trop}}_{2} = (\infty, 0 , \infty)\quad$};
\fill[white] (-2,0) circle (.045cm) node[black,align=left,   above]{};
\fill[black] (2,0) circle (.05cm) node[black,align=left,   above]{$(0, \infty, \infty) = e^{\text{trop}}_{1}\quad$};
\fill[white] (2,0) circle (.045cm) node[black,align=left,   above]{};
\fill[black] (-1,-2) circle (.05cm) node[black,align=left,   above]{};
\fill[white] (-1,-2) circle (.045cm) node[black,align=left,   above]{};
\fill[black] (1,-2) circle (.05cm) node[black,align=left,   above]{};
\fill[white] (1,-2) circle (.045cm) node[black,align=left,   above]{};
\fill[black] (0,-4) circle (.05cm) node[black,align=left,   below]{$e^{\text{trop}}_{3} = (\infty, \infty, 0)\quad$};
\fill[white] (0,-4) circle (.045cm) node[black,align=left,   below]{};
\end{tikzpicture}
\caption{A pictorial description of $\mathbb{T}\mathbb{P}^{2}$}
\label{fig:tropical_projective_plane}
\end{figure}


\end{eg}


We see from the description of $\mathbb{T}\mathbb{P}^{2}$ in Figure \ref{fig:tropical_projective_plane} that all the automorphisms of this space are just permutations of the coordinate directions. This complements our result in Theorem \ref{thm:tropicalfund}, describing the structure of each automorphism of the automorphism group of $\mathbb{T}\mathbb{P}^{2}$ as a conjugation of a semiring semi linear map of $\mathbb{T}^{2}$ and an automorphism of the semiring $\mathbb{T}$.

We highlight the fact that the space $\overline{V}$ in Theorem \ref{thm:alternatefund} is an example of a \emph{Tits Building}. We elaborate on this connection in Section \ref{sec:section5} and how we could explore possible tropical equivalents of classical buildings that might arise based on the results that we prove in this article. 


\section{Projectivity in the tropical plane}\label{sect:section4}

We recall some definitions from classical projective geometry concerning pencils of points and lines in the plane,

\begin{def0}
Given two lines $l_{1}$ and $l_{2}$ and a point $P$ not lying on both of them, a \textbf{perspectivity} is a bijective mapping between the pencils of points on $l_{1}$ and $l_{2}$, such that the lines incident with the corresponding points of the two pencils is concurrent at $P$. The point $P$ is referred as the center of the perspectivity.
\end{def0}

If $A,B$ and $C$ is the pencil of points on $l_{1}$ and $A',B'$ and $C'$ is the pencil of points on $l_{2}$ such that the point $P$ is the center of the perspectivity between them, then we represent this as 

\[  ABC \>\>     \stackrel{P}{\doublebarwedge} \>\> A'B'C'  \]
\begin{def0}
A finite composition of two or more perspectivities is called a \textbf{projectivity}.
\end{def0}

 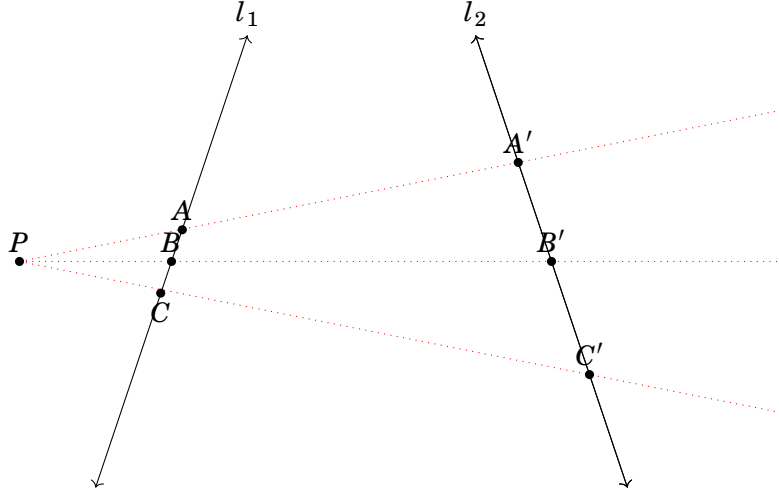
\begin{figure}[th]\centering
    \begin{tikzpicture}[scale=1]
      \draw[][<->] (1,3) -- (-1,-3);
      \draw[][<->] (4,3) -- (6,-3);
      \draw[][<->] (4,3) -- (6,-3);
      \draw[dotted,red][] (-2,0) -- (8,0);
      \draw[dotted, red][] (-2,0) -- (8,2);
      \draw[dotted, red][] (-2,0) -- (8,-2);
      \fill[black] (1,3) circle (.000000001cm) node[align=left,   above]{$l_{1}$};
      \fill[black] (4,3) circle (.0000001cm) node[align=left,   above]{$l_{2}$};
      \fill[black] (-2,0) circle (.06cm) node[align=left,   above]{$P$};
      \fill[black] (0.142,0.42) circle (.06cm) node[align=left,   above]{$A$};
      \fill[black] (0,0) circle (.06cm) node[align=right,   above]{$B$};
      \fill[black] (-0.142,-0.42) circle (.06cm) node[align=left,   below]{$C$};
      \fill[black] (4.5625,1.3125) circle (.06cm) node[align=left,   above]{$A'$};
      \fill[black] (5,0) circle (.06cm) node[align=left,   above]{$B'$};
      \fill[black] (5.5,-1.5) circle (.06cm) node[align=left,   above]{$C'$};
    \end{tikzpicture}
    \caption{A example of a classical perspectivity}
    \label{fig:classical_perspectivity}
  \end{figure}

We now consider the following classical result, termed as the fundamental theorem involving projectivity of lines in the plane \cite[Theorem 4.8]{CB04}

\begin{thm0}\label{thm:uniq_proj}
A projectivity between two pencils is uniquely determined by three pairs of corresponding points. 
\end{thm0}
 
This result rests on the following axiom \cite[Axiom 6]{CB04}

\begin{axi0}\label{axiom:ax1}
If a projectivity on a pencil of points leaves three distinct points of the pencil invariant, it leaves every point of the pencil invariant.
\end{axi0}

We now define the set of corresponding definitions in the tropical plane,




\begin{def0}
We say that $l(p,q,r)$ is a \textbf{tropical pencil} of points with a line $l$ and $n = p + q + r$ points, if $l$ is a tropical line with $p$ points on the $(1,1)$ half ray, $q$ points on the $(-1,0)$ half ray  and $r$ points on the $(0,-1)$ half ray. 
\end{def0}

\begin{def0}
We define $l(\overline{a},\overline{b},\overline{c})$ to be the \textbf{reduced tropical pencil} derived from a tropical pencil $l(p,q,r)$, where $l$ is the tropical line, along with one point out of all coaxial points on each of its half rays as a representative from the three families of coaxial points, as illustrated in the Figure \ref{fig:red_pencil}.
\end{def0}

 \begin{figure}[H]\centering
    \begin{tikzpicture}[scale=0.4]
      \draw[][->] (0,0) -- (-6,0);
      \draw[][->] (0,0) -- (0,-6);
      \draw[][->] (0,0) -- (6,6);
      
      \draw[][<->] (5,0) -- (7,0);
      \fill[blue] (-1,0) circle (.1cm) node[align=left,   above]{$a_{1}$};
      \fill[blue] (-2,0) circle (.1cm) node[align=left,   above]{$a_{2}$};
      \fill[blue] (-5,0) circle (.1cm) node[align=left,   above]{$a_{q}$};
      \draw[dotted][->] (-2.6,0.3) -- (-4.2,0.3);
      \fill[green] (0,-1) circle (.1cm) node[label=right:$b_{1}$]{};
      \fill[green] (0,-2) circle (.1cm) node[label=right:$b_{2}$]{};
      \fill[green] (0,-5) circle (.1cm) node[label=right:$b_{r}$]{};
      \draw[dotted][->] (0.3,-2.6) -- (0.3,-4.2);
      \fill[red] (1,1) circle (.1cm) node[align=left,   above]{$c_{1}$};
      \fill[red] (2,2) circle (.1cm) node[align=left,   above]{$c_{2}$};
      \fill[red] (5,5) circle (.1cm) node[align=left,   above]{$c_{p}$};
      \draw[dotted][->] (2.3,2.6) -- (4.6,4.9);

      \draw[][->] (16,0) -- (10,0);
      \draw[][->] (16,0) -- (16,-6);
      \draw[][->] (16,0) -- (22,6);
      
    \fill[blue] (13,0) circle (.2cm) node[align=left,   above]{$\overline{a}$};  
    \fill[green] (16,-3) circle (.2cm) node[label=right:$\overline{b}$]{};
    \fill[red] (19,3) circle (.2cm) node[align=left,   above]{$\overline{c}$};
    \end{tikzpicture}
    \caption{A reduced tropical pencil}
    \label{fig:red_pencil}
  \end{figure}
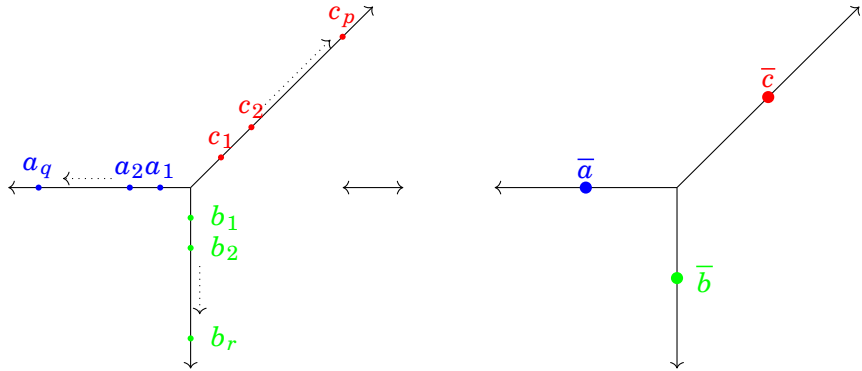

In Figure \ref{fig:red_pencil}, $\overline{a} = a_{i}$ for some $i$ such that $1 \leq i \leq q$, 
$\overline{b} = b_{j}$ for some $j$ such that $1 \leq j \leq r$ and $\overline{c} = c_{k}$ for some $k$ such that $1 \leq k \leq p$. We also observe that any reduced tropical pencil has at most three points. Essentially, a reduced tropical pencil is obtained by using coaxiality as an equivalence relation among points, and considering one representative from each equivalence class of coaxial points.

\begin{def0}
Two tropical pencils $l_{1}(p_{1},q_{1},r_{1})$ and $l_{2}(p_{2},q_{2},r_{2})$ are said to be \textbf{compatible} if the lines $l_{1}$ and $l_{2}$ are non-coaxial and $p_{1} = p_{2}, \> q_{2} = q_{2}, \> r_{2} = r_{2}$.
\end{def0}

 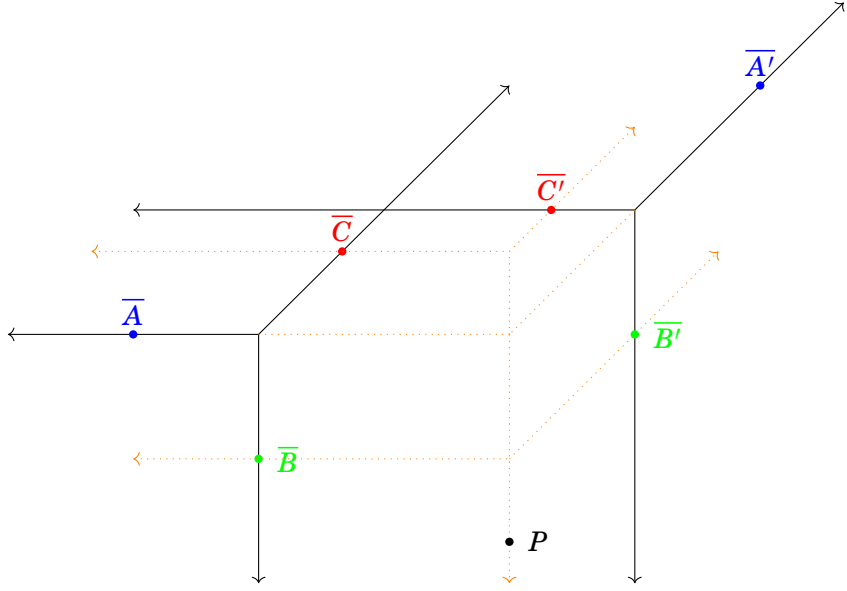
\begin{figure}[ht!]\centering
    \begin{tikzpicture}[scale=0.55]
      \draw[][->] (0,0) -- (-6,0);
      \draw[][->] (0,0) -- (0,-6);
      \draw[][->] (0,0) -- (6,6);
      
      \fill[blue] (-3,0) circle (.1cm) node[align=left,   above]{$\overline{A}$};
      \fill[green] (0,-3) circle (.1cm) node[label=right:$\overline{B}$]{};
      \fill[red] (2,2) circle (.1cm) node[align=left,   above]{$\overline{C}$};
      
      
      \draw[][<-] (-3,3) -- (9,3);
      \draw[][->] (9,3) -- (9,-6);
      \draw[][->] (9,3) -- (14,8);
      
      \draw[dotted,orange][<-] (-4,2) -- (6,2);
      \draw[dotted,orange][] (9,3) -- (6,0);
      \draw[dotted,orange][] (0,0) -- (6,0);
      \draw[dotted,orange][->] (6,0) -- (6,-6);
      \draw[dotted,orange][<-] (-3,-3) -- (6,-3);
      \draw[dotted,orange][->] (6,-3) -- (11,2);
      \draw[dotted,orange][] (6,0) -- (6,2);
      \draw[dotted,orange][->] (6,2) -- (9,5);
      
      \fill[green] (9,0) circle (.1cm) node[label=right:$\overline{B'}$]{};
      \fill[blue] (12,6) circle (.1cm) node[align=left,   above]{$\overline{A'}$};
      \fill[black] (6,-5) circle (.1cm) node[label=right:${P}$]{};
      \fill[red] (7,3) circle (.1cm) node[align=left,   above]{$\overline{C'}$};
    \end{tikzpicture}
    \caption{A tropical perspectivity}
    \label{fig:trop_perspec}
  \end{figure}

We now define tropical counterparts to perspectivity and projectivity,

\begin{def0}
Given two reduced tropical pencils $l(\overline{A},\overline{B},\overline{C})$ and $l'(\overline{A'},\overline{B'},\overline{C'})$ which are non-coaxial and a point $P$ not on both of them, a \textbf{tropical perspectivity} is a bijective mapping between the points on $l$ and $l'$, such that the tropical lines incident with the corresponding points of the two pencils is concurrent at $P$.
\end{def0}

In Figure \ref{fig:trop_perspec}, we see a tropical perspectivity which maps $\overline{A}$ to $\overline{A'}$, $\overline{B}$ to $\overline{B'}$ and $\overline{C}$ to $\overline{C'}$, i.e,

\[   \overline{A} \> \overline{B} \> \overline{C} \>\>     \stackrel{P}{\doublebarwedge} \>\> \overline{A'} \> \overline{B'} \> \overline{C'}   \] 

\begin{remark}
We observe that a tropical perspectivity always maps at most three points from one line to the other. 
\end{remark}

\begin{def0}
A \textbf{tropical projectivity} is a bijective mapping between compatible tropical pencils obtained as a finite composition of tropical perspectivities.
\end{def0}

We now establish a tropical counterpart to Theorem \ref{thm:uniq_proj},

\begin{thm0}\label{thm:uniq_proj_trop}
A tropical projectivity between two compatible tropical pencils is uniquely determined (up to equivalence) by three pairs of corresponding points. 
\end{thm0}

\begin{proof}

We consider two compatible tropical pencils  $l_{1}(p,q,r)$ and $l_{2}(p',q',r')$, and consider three points $A,B,C$ and $A', B', C'$ on $l_{1}$ and $l_{2}$ respectively such that the corresponding reduced tropical pencils are  $l_{1}(\overline{A},\overline{B},\overline{C})$ and $l_{2}(\overline{A'},\overline{B'},\overline{C'})$. We consider in our construction that the chosen six points lie on different rays of the given tropical line, however as we will see in the latter part of the proof, the argument for the other cases is also similar to the one we consider here. 

We begin by a constructive approach to show the desired projectivity exists and the argument for the uniqueness would follow. Consider the points $\overline{A},\overline{B}$ and $\overline{C}$ on $l_{1}$  and $\overline{A'},\overline{B'}$ and $\overline{C'}$ on $l_{2}$, as shown in Figure \ref{fig:perspectivity}. We construct the tropical line between the points $A$ and $A'$ and we choose a point $p'$ on this line. Let $l'$ be an arbitrary line passing through $A'$. Then $B_{1} = Bp' \cdot l' $, i.e., $B_{1}$ is the point of stable intersection between the lines $l'$ and the tropical line passing through $B$ and $p'$. We observe that there exists a tropical perspectivity between $l_{1}$ and $l'$, defined as follows  

\begin{equation}\label{eq:perspectivity1}
  ABC \>\>  \stackrel{p'}{\doublebarwedge} \>\> A'B_{1}C      
\end{equation}

where $A$ is mapped to $A'$, $B$ is mapped to $B_{1}$ and $C$ is mapped to itself via $p'$.

Subsequently, we consider the point $p' = CC' \cdot B_{1}B'$, a point lying on the intersection between the tropical line passing through $C$ and $C'$ and the tropical line passing through $B_{1}$ and $B'$. We see that we obtain another tropical perspectivity via $p'$ where

\begin{equation}\label{eq:perspectivity2}
  A'B_{1}C \>\>  \stackrel{p"}{\doublebarwedge} \>\> A'B'C'    
\end{equation}

 $A$ is mapped to itself, $B_{1}$ is mapped to $B'$ and $C$ is mapped to itself via $C'$.

Using the equations \ref{eq:perspectivity1} and \ref{eq:perspectivity2} we obtain a tropical projectivity,

\begin{equation}
  ABC \>\> \barwedge \>\> A'B'C'   
\end{equation}

where the $p$ points on the $(1,1)$ ray of $l_{1}$ are mapped to the $p' = p$ points on the $(1,1)$ ray of $l_{2}$ via the reduced pencil and similarly the $q$ points on the $(-1,0)$ ray of $l_{1}$ are mapped to the $q' = q$ points on the $(-1,0)$ ray of $l_{2}$ and the $r$ points on the $(0,-1)$ ray of $l_{1}$ are mapped to the $r' = r$ points on the $(0,-1)$ ray of $l_{2}$. Since, the choice of the six points $A,B,C,A'B'$ and $C'$ fixes the projectivity, we have the desired result.

 \begin{figure}\centering
    \begin{tikzpicture}[scale=0.75]
      \draw[][->] (0,0) -- (-4,0);
      \draw[][->] (0,0) -- (0,-4);
      \draw[][->] (0,0) -- (6,6);
      \draw[dotted,blue][] (0,0) -- (5,0);
      \draw[dotted,blue][->] (5,0) -- (5,-4);
      \draw[dotted,blue][->] (5,0) -- (7,2);
      \draw[dotted,blue][->] (5,-2) -- (-4,-2);
      \draw[dotted,blue][->] (5,-2) -- (12,5);
      \draw[][<-] (-4,2) -- (5,2);
      \draw[dotted,blue][] (5,2) -- (5,0);
      \draw[dotted,blue][->] (5,2) -- (9,6);
      \draw[][] (5,2) -- (7,2);
      \draw[][->] (7,2) -- (7,-4);
      \draw[][->] (7,2) -- (11,6);
      \fill[black] (0,0) circle (.0000001cm) node[align=left,   above]{$l_{1}$};
      \fill[blue] (-2,0) circle (.1cm) node[align=left,   above]{$A$};
      \fill[green] (0,-2) circle (.1cm) node[label=right:$B$]{};
      \fill[red] (2,2) circle (.1cm) node[align=right,   above]{$C$};
      \fill[blue] (9,4) circle (.1cm) node[align=right,   above]{$A'$};
      \fill[green] (7,0) circle (.1cm) node[align=right,   above]{$B_{1}$};
      \fill[black] (5,-3) circle (.1cm) node[align=right,   above]{$p'$};
      \fill[green] (12,2) circle (.1cm) node[label=right:$B'$]{};
      \fill[red] (15,7) circle (.1cm) node[align=right,   above]{$C'$};
      \fill[black] (10,-3) circle (.1cm) node[align=left,   above]{$p"$};
      \fill[black] (12,4) circle (.0000001cm) node[align=left,   above]{$l_{2}$};
      \fill[black] (11,6) circle (.000001cm) node[align=left,   above]{$l'$};
      \fill[black] (7,-4) circle (.000001cm) node[align=left,   below]{$l'$};
      \fill[black] (-4,2) circle (.000001cm) node[align=left,   above]{$l'$};
      \draw[][->] (12,4) -- (-2,4);
      \draw[][->] (12,4) -- (12,-4);
      \draw[][->] (12,4) -- (16,8);
      \draw[dotted,red][] (12,4) -- (10,2);
      \draw[dotted,red][] (7,2) -- (10,2);
      \draw[dotted,red][->] (10,2) -- (10,-4);
      \draw[dotted,red][] (5,0) -- (10,0);
      \draw[dotted,red][->] (10,0) -- (14,4);
      \draw[dotted,red][] (10,2) -- (10,4);
      \draw[dotted,red][->] (10,4) -- (12,6);
    \end{tikzpicture}
    \caption{Tropical projectivity between two reduced tropical pencil of points}
    \label{fig:perspectivity}
  \end{figure}
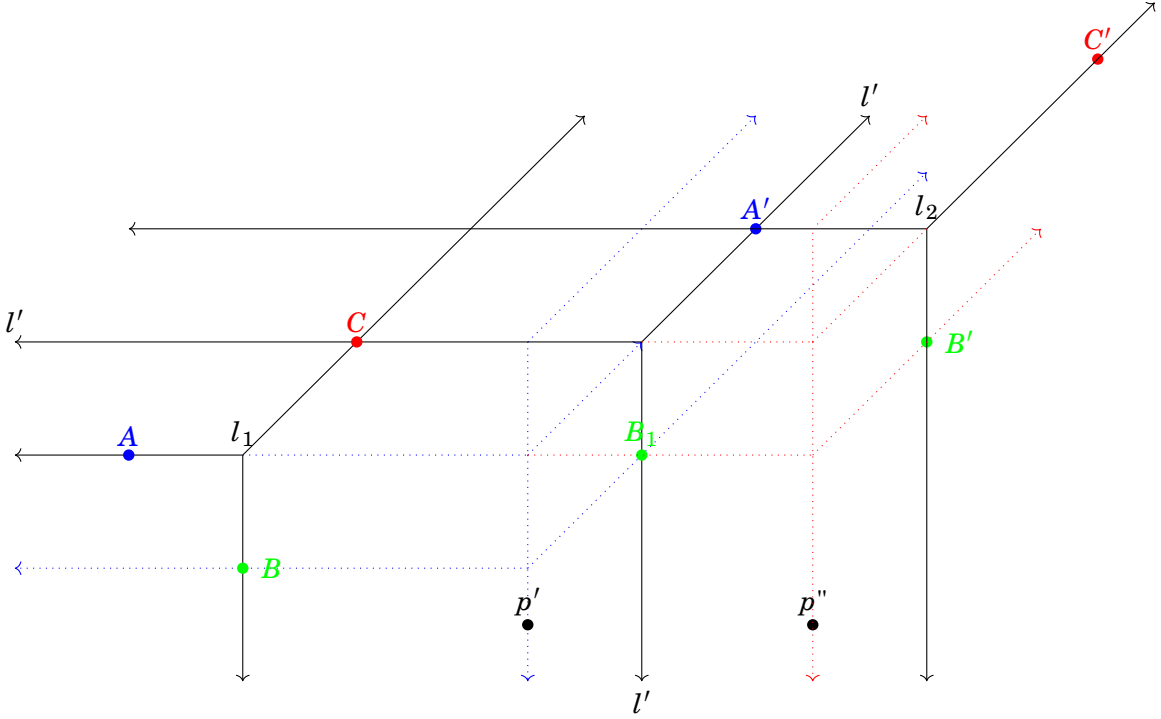
  
\end{proof}

\section{Conclusions}\label{sec:section5}

We would like to elaborate on the various connections and possible extensions of our results presented here. As alluded to earlier, restricted to the case of the tropical semiring, study regarding affine buildings has been carried out in \cite{joswig2007affine}, \cite{werner2011tropical}. Since, the fundamental theorem of projective geometry is a structure theorem for classical groups, hence it plays an important role in the study of classical groups and associated buildings. Therefore, a natural successor of our results could be a possible definition of classical groups over semimodules and associated buildings. 

We recall some basic definitions from \cite{joswig2014essentials} and \cite{joswig2007affine}. In this case, the study focuses on the case when the underlying field is $\mathbb{K} = \mathbb{C}((t)) $, field of Laurent series with complex coefficients. It has a valuation ring which is local, with the residue field isomorphic to $\mathbb{C}$. The vector space $\mathbb{K}^{d}$ is a module over the valuation ring and lattices over this ring are considered as submodules generated by $d$ linearly independent vectors in $\mathbb{K}^{d}$. Two lattices $\lambda_{1}, \lambda_{2} \subset \mathbb{K}^{d}$ are called \emph{equivalent} if $c \cdot \lambda_{1} = \lambda_{2}$, for some $c \in K^{x}$. Two equivalence class of lattices $\langle \lambda_{1} \rangle$ and $\langle \lambda_{2} \rangle$ are called \emph{adjacent} if there are representatives $\lambda_{1}$ and $\lambda_{2}$ such that $t \cdot \lambda_{2} \subset \lambda_{1} \subset \lambda_{2}$.

The adjacency relation on equivalence classes of lattices in $\mathbb{K}_{d}$ defines a graph,
and its \emph{flag simplicial complex} is the \emph{Bruhat–Tits building} $\mathbb{B}_{d}(\mathbb{K})$. A point in the tropical projective space $\mathbb{TP}^{d-1}$ is a \emph{lattice point} if it is represented by a vector $x$ in $\mathbb{Z}^{d}$. A \emph{tropical lattice polytope} is the tropical convex hull of finitely many lattice points in $\mathbb{TP}^{d-1}$. Similar to the classical case, we can define adjacency between tropical lattice points with respect to the \emph{natural metric} defined in \cite{joswig2007affine}. This defines a graph whose flag simplicial complex provides a \emph{triangulation} of the affine space $\mathbb{TP}^{d-1}$.

We believe that similar to the case of modules over valuation rings, one can study lattices over semimodules over an appropriately chosen semiring, much in essence to the study which we carried out in this article. This would also help in defining buildings over semimodules, and study of matrix groups over semirings. Further details will be explored elsewhere.

We also would like to point the reader to the recent developments in the theory of \emph{ordered blueprints} \cite{baker2018moduli} and possible connections to our work. We recall some basic definitions from \cite{baker2018moduli}. An \emph{ordered semiring} is a semiring $R$ with a partial order $\leq$ such that $x \leq y$ and $z \leq t$ implies $x + z \leq y + t$ and $xz \leq yt$. An \emph{ordered blueprint} is a triple $B = (B^{.}, B^{+}, \leq)$ where $(B^{+}, \leq)$ is an ordered semiring and $B^{.}$ is a multiplicative subset of $B^{+}$ that generates $B^{+}$ as a semiring and contains $0$ and $1$. An \emph{ordered bluefield} is an ordered blueprint $B$ with $B^{.} = B^{\times} \cup \{ 0 \}$. We suggest the reader to refer \cite{baker2018moduli} for further elaborate descriptions of \emph{ordered blue schemes} and the construction of the functor \emph{Proj}, but we emphasize the main takeaway; the notion of a projective space over ordered blueprints, and examples of projective line and projective planes \cite[Example 4.13]{baker2018moduli}. This leads us to the following question, 

\begin{question}
Can a version of fundamental theorem of projective geometry be established over matroids in the background of ordered blueprints and the associated projective space ?
\end{question}

We wish to explore this question further and connect with the existing notion of projective geometry over ordered blueprints.



We now discuss one of the most prominent aspects of classical projective geometry related to our results - \emph{cross ratios}. Most of what we recall here is referenced from the book by Jürgen Richter-Gebert \cite{richter2011}. We recall the notation that a point $a \in \mathbb{R}^{2}$ is denoted as $ a =   \begin{pmatrix}
  a_{1} \\
  a_{2} \\
\end{pmatrix} $  
and
\[ [a,b] = \text{det}\begin{pmatrix}
a_{1} & b_{1} \\
a_{2} & b_{2} 
\end{pmatrix} \]

\begin{def0}
Let $a,b,c,d$ be four nonzero vectors in $\mathbb{R}^{2}$. The \emph{cross-ratio} $(a,b;c,d)$ is the following quantity,

\[ (a,b;c,d) := \frac{[a,c][b,d]}{[a,d][b,c]}  \]
\end{def0}

This definition extends naturally to complex numbers and also to arbitrary fields. An important feature of cross ratios is that it is invariant under projective transformations.

\begin{lemma}
For any real nonzero parameters $\lambda_{a},\lambda_{b},\lambda_{c},\lambda_{d} \in \mathbb{R}$ we have 

\[(a,b;c,d) = (\lambda_{a}a,\lambda_{b}b;\lambda_{c}c,\lambda_{d}d)\]
\end{lemma}

\begin{lemma}\label{lemma:matrix}
Let $M$ be a $2 \times 2$ matrix with nonvanishing determinant and let $a,b,c,d$ be four vectors in $\mathbb{R}^{2}$. Then we have 

\[(a,b;c,d) = (M \cdot a, M\cdot b; M \cdot c, M \cdot d)\]
\end{lemma}

This invariance can be seen in the following result regarding classical perspectivities in the plane,

\begin{cor0}[Corollary 4.1 \cite{richter2011}]
Let $o$ be a point and let $l$ and $l'$ be two lines not passing through $o$. If four points $a,b,c,d$ are on $l$ and are projected by the viewpoint $o$ to four points $a',b',c',d'$ on $l'$, then the cross ratios satisfy $(a,b;c,d) = (a',b';c',d')$.
\end{cor0}

We now explore the possibility of a tropical counterpart to the notion of a cross ratio. Firstly, we recall the definition of the \emph{tropical determinant} \cite{maclagan2009introduction}. For a $n \times n$ matrix $X = x_{ij}$, with entries from the tropical semiring, the tropical determinant is defined as follows

\begin{equation*}
\begin{split}
 \text{tdet}(X) & = \bigoplus_{\sigma \in S_{n}}  x_{1\sigma(1)} \odot x_{2\sigma(2)} \odot \hdots \odot x_{n\sigma(n)} \\
& = \max_{\sigma \in S_{n}}\ x_{1\sigma(1)} \odot x_{2\sigma(2)} \odot \hdots \odot x_{n\sigma(n)}
\end{split}    
\end{equation*}

where $\sigma$ represents a cycle in $S_{n}$, the symmetric group on $n$ elements. A square matrix $X$ is said to be \emph{tropically singular} if $\text{tdet(X)} = \infty$ or the tropical polynomial tdet vanishes on $X$ \cite{joswig2014essentials}.

We fix the following notation, 

\[ [a,b]^{tr} = \text{tdet}\begin{pmatrix}
a_{1} & b_{1} \\
a_{2} & b_{2} 
\end{pmatrix} \]

Let $a,b,c,d$ be four nonzero vectors in $\mathbb{R}^{2}$. We consider the  
the following quantity,

\[ (a,b;c,d)^{tr} := ([a,c]^{tr} \odot [b,d]^{tr}) -  ([a,d]^{tr} \odot [b,c]^{tr})  \]

We now try to see if $(a,b;c,d)^{tr}$ can be a candidate for a \emph{tropical cross-ratio}.
We show invariance of $(a,b;c,d)^{tr}$ with respect to tropical scalar multiplication,

\begin{lemma}
For any real nonzero parameters $\lambda_{a},\lambda_{b},\lambda_{c},\lambda_{d} \in \mathbb{R}$ we have 

\[(a,b;c,d)^{tr} = (\lambda_{a} \odot a,\lambda_{b} \odot b;\lambda_{c} \odot c,\lambda_{d} \odot d)^{tr}\]
\end{lemma}

\begin{proof}

\begin{equation*}
\begin{split}
(\lambda_{a} \odot a,\lambda_{b} \odot b;\lambda_{c} \odot c,\lambda_{d} \odot d)^{tr} & = \left( \text{tdet}\begin{pmatrix}
\lambda_{a} \odot a_{1} & \lambda_{c} \odot c_{1} \\
\lambda_{a} \odot a_{2} & \lambda_{c} \odot c_{2} 
\end{pmatrix} \odot \text{tdet}\begin{pmatrix}
\lambda_{b} \odot b_{1} & \lambda_{d} \odot d_{1} \\
\lambda_{b} \odot b_{2} & \lambda_{d} \odot d_{2} 
\end{pmatrix} \right) \\ 
& - \left( \text{tdet}\begin{pmatrix}
\lambda_{a} \odot a_{1} & \lambda_{d} \odot d_{1} \\
\lambda_{a} \odot a_{2} & \lambda_{d} \odot d_{2} 
\end{pmatrix} \odot \text{tdet}\begin{pmatrix}
\lambda_{b} \odot b_{1} & \lambda_{c} \odot c_{1} \\
\lambda_{b} \odot b_{2} & \lambda_{c} \odot c_{2} 
\end{pmatrix} \right) \\
& = \text{max}(\lambda_{a} \odot a_{1} + \lambda_{c} \odot c_{2} , \lambda_{a} \odot a_{2} + \lambda_{c} \odot c_{1}) + \text{max}(\lambda_{b} \odot b_{1} + \lambda_{d} \odot d_{2}, \\
& \lambda_{b} \odot b_{2} + \lambda_{d} \odot d_{1}) - (\text{max}(\lambda_{a} \odot a_{1} + \lambda_{d} \odot d_{2}, \\
& \lambda_{a} \odot a_{2} + \lambda_{d} \odot d_{1}) + \text{max}(\lambda_{b} \odot b_{1} + \lambda_{c} \odot c_{2}, \lambda_{b} \odot b_{2} + \lambda_{c} \odot c_{1}) ) \\
& = \lambda_{a} + \lambda_{c} + \text{max}(a_{1} + c_{2}, a_{2} + c_{1}) + \lambda_{b} + \lambda_{d} + \text{max}(b_{1} + d_{2}, b_{2} + d_{1}) - \\
& (\lambda_{a} + \lambda_{d} + \text{max}(a_{1} + d_{2}, a_{2} + d_{1}) + \lambda_{b} + \lambda_{c} + \text{max}(b_{1} + c_{2}, b_{2} + c_{1})) \\
& = (\lambda_{a} + \lambda_{b} + \lambda_{c} + \lambda_{d}) + \text{max}(a_{1} + c_{2}, a_{2} + c_{1}) + \text{max}(b_{1} + d_{2}, b_{2} + d_{1}) \\
& -(\lambda_{a} + \lambda_{b} + \lambda_{c} + \lambda_{d}) - \text{max}(a_{1} + d_{2}, a_{2} + d_{1}) - \text{max}(b_{1} + c_{2}, b_{2} + c_{1})) \\
& = (a,b;c,d)^{tr}
\end{split}
\end{equation*}
\end{proof}

However, we are not able to recover all classical results. For example, let us consider the case of Lemma \ref{lemma:matrix}. We consider a tropically non-singular $2 \times 2$ matrix $M = \begin{pmatrix}
m_{1} & m_{2} \\
m_{3} & m_{4} 
\end{pmatrix}$. Then for $ a =   \begin{pmatrix}
  a_{1} \\
  a_{2} \\
\end{pmatrix} $ and $ b =   \begin{pmatrix}
  b_{1} \\
  b_{2} \\
\end{pmatrix} $ we get, 

\[ M \odot a =  \begin{pmatrix}
m_{1} \odot a_{1} \oplus  m_{2} \odot a_{2} \\
m_{3} \odot a_{1} \oplus  m_{4} \odot a_{2} 
\end{pmatrix} \quad \text{and} \quad  M \odot b =  \begin{pmatrix}
m_{1} \odot b_{1} \oplus  m_{2} \odot b_{2} \\
m_{3} \odot b_{1} \oplus  m_{4} \odot b_{2} 
\end{pmatrix} \]

\begin{equation*}
\begin{split}
[M \odot a, M \odot b]^{tr} & = \text{tdet}\begin{pmatrix}
m_{1} \odot a_{1} \oplus  m_{2} \odot a_{2} & m_{1} \odot b_{1} \oplus  m_{2} \odot b_{2} \\
m_{3} \odot a_{1} \oplus  m_{4} \odot a_{2}  & m_{3} \odot b_{1} \oplus  m_{4} \odot b_{2} 
\end{pmatrix} \\
& = \underline{m_{1} \odot m_{3} \odot a_{1} \odot b_{1}} \oplus m_{1} \odot m_{4} \odot a_{1} \odot b_{2} m_{2} \odot m_{3} a_{2} \odot b_{1} \oplus \underline{m_{2} \odot m_{4} \odot a_{2} \odot b_{2}} \\
& \oplus \underline{m_{1} \odot m_{3} \odot a_{1} \odot b_{1}} \oplus m_{3} \odot m_{2} \odot a_{1} \odot b_{2} \oplus m_{1} \odot m_{4} \odot a_{2} \odot b_{1} \oplus \underline{m_{2} \odot m_{4} \odot a_{2} \odot b_{2}} \\
&  = a_{1} \odot b_{2} \odot ( m_{1} \odot m_{4} \oplus m_{2} \odot m_{3} ) \oplus a_{2} \odot b_{1} \odot ( m_{1} \odot m_{4} \oplus m_{2} \odot m_{3} ) \\
& \oplus m_{1} \odot m_{3} \odot a_{1} \odot b_{1} \oplus m_{2} \odot m_{4} \odot a_{2} \odot b_{2} \\
& = ( m_{1} \odot m_{4} \oplus m_{2} \odot m_{3} ) \odot (a_{2} \odot b_{1} \oplus a_{1} \odot b_{2}) \oplus m_{1} \odot m_{3} \odot a_{1} \odot b_{1} \oplus m_{2} \odot m_{4} \odot a_{2} \odot b_{2} \\
& = \text{tdet}(M) \cdot [a,b]^{tr} \oplus m_{1} \odot m_{3} \odot a_{1} \odot b_{1} \oplus m_{2} \odot m_{4} \odot a_{2} \odot b_{2}
\end{split}    
\end{equation*}
We see that in the tropical case, we have an additional term along with $\text{tdet}(M) \cdot [a,b]^{tr}$, which is due to the fact that there is no additive inverse tropically and the tropical determinant is a tropical sum whereas the classical determinant has a change of sign for terms which results in cancellations. Therefore, for a tropically non-singular $2 \times 2$ matrix $M$ we cannot conclude that 

\[(a,b;c,d)^{tr} = (M \cdot a, M\cdot b; M \cdot c, M \cdot d)^{tr} \]

We would like to explore other suitable variants of the definition of a tropical cross ratio, in the context of tropical projective geometry, which complement the definition of tropical perpsectivity and remain invariant under projective transformations, which is true in the classical case.

\begin{remark}
We do acknowledge that there is a related description of  tropical double ratios and tropical cross ratio in the context of curve counting discussed in \cite{MR2404949} and \cite{goldner2021counting} respectively. Also, in \cite{baker2018moduli} definition of a cross ratio is defined over the basis-exchange graph of a matroid over a pasture. However, for our purposes of looking at tropical projective geometry, we consider other variants which are more suited for our discussion.  
\end{remark}






\bibliographystyle{alpha}
\bibliography{biblio.bib}
\end{document}